\documentclass[thmsa,ukenglish]{article}
\usepackage{amsmath}
\usepackage{graphicx}
\usepackage[all,rotate]{xy}
\usepackage{amsfonts}

\usepackage{amssymb}
\usepackage{url}
\usepackage{verbatim}
\usepackage{hyperref}

\newtheorem{theorem}{Theorem}[section]

\newtheorem{corollary}[theorem]{Corollary}

\newtheorem{definition}[theorem]{Definition}

\newtheorem{lemma}[theorem]{Lemma}

\newtheorem{proposition}[theorem]{Proposition}

\newenvironment{proof}[1][Proof]{\textbf{#1.} }{\ \rule{0.5em}{0.5em}}

\newcommand{\up}{\mathord{\uparrow}}
\newcommand{\lup}{\mathord{\upharpoonleft}}
\newcommand{\rup}{\mathord{\upharpoonright}}
\newcommand{\lexl}{\leqslant}               
\newcommand{\lexu}{\eqslantless}            
\newcommand{\overlap}{\between}             %
\newcommand{\psign}{\mathord{+}}
\newcommand{\msign}{\mathord{-}}
\newcommand{\true}{1}                       
\newcommand{\false}{0}                      
\newcommand{\lmid}{\mathrel{<\negmedspace\mid}}  
\newcommand{\midl}{\mathrel{\mid\negmedspace<}}  
\newcommand{\half}{\mathord{\mathsf{half}}} 
\newcommand{\Ball}{\mathord{\mathrm{Ball}}} 
\newcommand{\ball}{\mathord{\mathrm{ball}}} 
\newcommand{\Idl}{\mathord{\mathrm{Idl}}}   

\newcommand{\Fr}{\mathbf{Fr}}               
\newcommand{\PreFr}{\mathbf{PreFr}}         
\newcommand{\Loc}{\mathbf{Loc}}             

\newcommand{\bigdvee}{\mathop{\bigvee\nolimits^\uparrow}}  
\newcommand{\bigdsqcup}{\mathop{\bigsqcup\nolimits^\uparrow}}  

\newcommand{\interval}{\mathbb{I}}         

\begin{document}

\title{The localic compact interval is an Escard\'{o}-Simpson interval object}
\author{Steven Vickers\\School of Computer Science, University of Birmingham,\\Birmingham, B15 2TT, UK.\\email: s.j.vickers@cs.bham.ac.uk}
\maketitle
\begin{abstract}
  The locale corresponding to the real interval $[-1,1]$ is an interval object, in the sense
  of Escard\'{o} and Simpson, in the category of locales.
  The map $c\colon 2^\omega \to [-1,1]$, mapping a stream $s$ of signs $\pm 1$ to
  $\Sigma_{i=1}^\infty s_i 2^{-i}$, is a proper localic surjection;
  it is also expressed as a coequalizer.
\end{abstract}

\section{Introduction}
In~\cite{EscardoSimpson:UniCCEI}, Escard\'{o} and Simpson prove a universal property for
the real interval $[-1,1]$, using a theory they develop of \emph{midpoint algebras}:
sets equipped with a binary operation that, abstractly, provides the midpoint of any two elements.
In an \emph{iterative} midpoint algebra there are also some limiting processes,
and it becomes possible there to define arbitrary convex combinations of two elements.
This property is expressed by saying that the interval $[-1,1]$ is freely generated,
as an iterative midpoint algebra, by its endpoints.
That is the universal property, and it thus characterizes the interval in a way that does not
explicitly describe the structure of reals.

The aim of this note is to prove an analogous property for the \emph{locale} $[-1,1]$ of Dedekind reals,
which we shall write $\interval$, in the category
$\mathbf{Loc}$ of locales.

The layout of the paper can be summarized section by section as follows.

Section~\ref{sec:ItMidptAlg} recalls midpoint algebras.

Section~\ref{sec:Cantor} develops some preliminary results on Cantor space $2^\omega$.
Principally, we analyse its localic presentation in order to get it in a
``join stable'' form suitable for the preframe coverage theorem,
a technical result used in Section~\ref{sec:ProperSurjn}.

Section~\ref{sec:ConvBody} shows as its main result that the interval $\interval$ is iterative.
Our proof relies on its metric structure,
and its embedding as the maximal points of a ``ball domain''.
The result of the iteration is then got via approximations in the ball domain.

Section~\ref{sec:c} introduces a map $c\colon 2^\omega\to\interval$ that can be
understood as the evaluation of infinite binary expansions.
We calculate some features of its inverse image function;
these results are needed in Section~\ref{sec:ProperSurjn}.

Section~\ref{sec:ProperSurjn} shows that $c$ is a localic surjection.
In fact it goes further and proves that it is a \emph{proper} surjection.
This is an essential part of the proof technique,
and also some such condition is needed to show later that $c\times c$
is also a surjection.
In essence this is a conservativity result:
to reason about real numbers it suffices to reason about the infinite binary expansions,
and this holds even in the absence of choice principles allowing one to choose an expansion
for every (Dedekind) real.
To prove it we use the preframe coverage theorem,
relying on the analysis of Sections~\ref{sec:Cantor} and~\ref{sec:c}.

Section~\ref{sec:Coequ} describes $c$ as the coequalizer of two maps from $2^\ast$ to $2^\omega$.

Section~\ref{sec:Initiality} now completes the proof that $\interval$ is an interval object.
Suppose we are given an iterative $A$ with two specified points as in Definition~\ref{def:intervalObject}~(3),
and we want to define the unique $N\colon\interval\to A$.
We find that $Nc = M$ (say) is easy to find,
so the task is to factor $M$ via $c$.
The unique existence of the factorization will follow from the coequalizer property of $c$.
It remains to show that $N$ preserves midpoints, and for this it is convenient to introduce
$3^\omega$, for streams of signs and zeros.

\section{Iterative midpoint algebras}\label{sec:ItMidptAlg}

We recall the definitions from \cite{EscardoSimpson:UniCCEI}, in an arbitrary
category with finite products.

\begin{definition}\label{def:midptAlg}
A \emph{midpoint algebra} is an object $A$ equipped with a morphism
$m\colon A\times A\rightarrow A$ satisfying the following conditions:%
\begin{align*}
m(x,x)  &  =x\\
m(x,y)  &  =m(y,x)\\
m(m(x,y),m(z,w))  &  =m(m(x,z),m(y,w))
\end{align*}

A homomorphism of midpoint algebras is a morphism that preserves the midpoint operation.

A midpoint algebra is \emph{cancellative} if it satisfies%
\[
m(x,z)=m(y,z)\Longrightarrow x=y\text{.}%
\]
\end{definition}

\begin{definition}\label{def:itMidptAlg}\label{def:convexBody}
A midpoint algebra $A$ is \emph{iterative} if, for every
object $X$ and pair of morphisms $h\colon X\rightarrow A$,
$t\colon X\rightarrow X$ (\emph{head} and \emph{tail}),
there is a unique morphism $M\colon X\rightarrow A$
making $M(x)=m(h(x),M(t(x))$ -- in other words, the following diagram commutes.%
\[
  \xymatrix{
    {A\times X}
      \ar@{->}[r]^{A\times M}
    & A\times A
      \ar@{->}[d]^{m}
    \\
    {X}
      \ar@{->}[u]^{\langle h,t\rangle}
      \ar@{->}[r]_{M}
    & {A}
  }
\]

A \emph{convex body} is a cancellative, iterative midpoint algebra.
\end{definition}

To illustrate the ``iterative'' condition,
a particular case would be where $X=\mathbb{N}$ and $t$ is the successor function.
Then $h$ is a sequence $(h_i)_{i\in\mathbb{N}}$.
In an affine setting, we would then have that $M(n)$ is the infinitary convex combination
\[
  M(n) = \sum_{i=n}^{\infty}\frac{1}{2^{i-n+1}}h_i
  \text{.}
\]

We now specialize to the category $\Loc$ of locales. The closed
Euclidean interval $\interval=[-1,1]$ is a cancellative midpoint algebra with
$m(x,y)=\frac{x+y}{2}$.
We shall think of the discrete two-point space $2$ as $\{\msign,\psign\}$,
so that Cantor space $2^{\omega}$ is the space of infinite sequences
(or \emph{streams}) of signs.

We also write $2^\ast$ for the set of finite sequences of signs,
$\varepsilon$ for the empty sequence,
$\sqsubseteq$ for the prefix order
and $|s|$ for the length of $s$.
We use juxtaposition to denote concatenation.

\begin{definition}\label{def:MPlusMinus}
  Suppose $A$ is an iterative midpoint algebra equipped with two points $a_{\psign}$ and $a_{\msign}$.
  We define $M_{a_{\msign}a_{\psign}}\colon 2^{\omega}\rightarrow A$ as the unique map such that%
  \[
    M_{a_{\msign}a_{\psign}}(\pm s)=m(a_{\pm},M_{a_{\msign}a_{\psign}}s)\text{.}%
  \]

  Referring to Definition~\ref{def:convexBody}, $X$ is $2^{\omega}$ and $h,t$
  are such that $\langle h,t\rangle(\pm s)=(a_{\pm},s)$ (so $t$ is the tail map
  in the usual sense).
\end{definition}

\begin{definition}\label{def:intervalObject}
  An \emph{interval object} $I$ is a free iterative
  midpoint algebra over 2. That is to say:

  \begin{enumerate}
  \item
    $I$ is equipped with two points $x_{\msign}$ and $x_{\psign}$ (its \emph{endpoints}).
  \item
    $I$ is an iterative midpoint algebra.
  \item
    For every iterative midpoint algebra $A$ with points $a_{\msign}$ and $a_{\psign}$
    there is a unique midpoint homomorphism $N\colon I\rightarrow A$ that
    takes $x_{\msign}$ and $x_{\psign}$ to $a_{\msign}$ and $a_{\psign}$ respectively.
  \end{enumerate}
\end{definition}

We shall prove that $\interval$, with endpoints $-1$ and $1$, is an interval object.

\section{Preliminary remarks on Cantor space}\label{sec:Cantor}

We take Cantor space $2^\omega$ to be the localic exponential of the discrete
locales $2$ (two points $\psign$ and $\msign$)
and $\mathbb{N}$ (natural numbers $1,2,3,\ldots$).
\footnote{There is a technical reason here for preferring to start at 1,
in that the first term in an infinite binary expansion is for $2^{-1}$.
For finite sequences too, the indexes will start at 1.}
This certainly exists, since discrete locales are locally compact. Its
(generalized) points can be described as the functions from $\mathbb{N}$ to $2$, and so
the frame can be presented by generators and relations as%
\begin{align*}
  \Fr\langle(n,\sigma)\in \mathbb{N}\times2\mid
       (n,\psign)\wedge (n,\msign) & \leq \false\\
                             \true &  \leq(n,\psign)\vee(n,\msign)\rangle
  \text{.}%
\end{align*}
(Here, abstractly, we write $\true$ and $\false$ for the top and bottom of a frame.
Where the locale has a definite name $X$, we shall also often write them as $X$ and $\emptyset$.)
Every generator $(n,\pm)$ has a Boolean complement $(n,\mp)$, so the locale is Stone.
Its frame is the ideal completion of the free Boolean algebra
on countably many generators $(n,\psign)$.

A little calculation shows that the frame is isomorphic to%
\begin{align*}
  \Fr\langle\up s\text{ (}s\in2^{\ast}\text{)} \mid
     \up t  &  \leq \up s\text{ (if }s\sqsubseteq t\text{)}\\
     \true  &  \leq \up\varepsilon\\
     \up s\wedge \up t   &   \leq \false \text{ (if }s,t\text{ incomparable)}\\
     \up s  &  \leq \up(s\msign)\vee \up(s\psign)\rangle\text{.}%
\end{align*}
The isomorphisms are given by%
\begin{align*}
\up s  &  \mapsto\bigwedge_{i=1}^{|s|}(i,s_{i})\\
(n,\sigma)  &  \mapsto\bigvee_{|s|=n-1}\up(s\sigma)\text{.}%
\end{align*}
The generators $\up s$ form a base.
$\up s$ comprises those streams of which $s$ is a prefix.

Later we shall need a preframe base, in other words opens of which every other
open is a directed join of finite meets, and for this we shall introduce
subbasics $\rup s$ and $\lup s$ that involve the
lexicographic ordering. Let us first introduce some notation.

\begin{definition}
If $s,t\in2^{\ast}$ then we write $s<t$ if there is some $u$ such that
$u\msign \sqsubseteq s$ and $u\psign \sqsubseteq t$.
We say that $s$ and $t$
\emph{differ} if either $s<t$ or $t<s$:
this is equivalent to their being incomparable under $\sqsubseteq$.
The relation $<$ extends to an open
$\bigvee_{u\in 2^\ast}\left(\up(u\msign)\times\up(u\psign)\right)$ of $2^{\omega}\times2^{\omega}$.

We write $s\lexl t$ if either $s<t$ or $s\sqsubseteq t$. This is just the
lexicographic order in which $\msign$ is less than $\psign$.

We write $s\lexu t$ if either $s<t$ or $t\sqsubseteq s$: in other
words, $t$ precedes $s$ in the dual lexicographic order with $\psign$ less than $\msign$.

Both $\lexl$ and $\lexu$ can be extended in the obvious way to the case where $s$ or $t$ may be infinite.

If $s\in 2^{\ast}$, then we define a \emph{right bristle} of $s$ to be a finite
sequence $t\psign$ such that $t\msign \sqsubseteq s$,
in other words a $u$ that is minimal (under $\sqsubseteq$) subject to $s<u$.
Dually, a \emph{left bristle} of $s$ is a $u$ minimal subject to $u<s$.
\end{definition}

\begin{definition}\label{def:leftRightHook}
If $s\in2^{\ast}$ then we define the open
$\rup s$ of $2^{\omega}$ as the finite join $\up s\vee
\bigvee\{\up t\mid t$ a right bristle of $s\}$. It comprises those $u$ in
$2^{\omega}$ such that $s\lexl u$.
Dually, we define
$\lup s= \up s\vee\bigvee\{\up t\mid t\text{ a left bristle of }s\}$,
comprising those $u$ such that $u\lexu s$.
\end{definition}

\begin{lemma}\label{lem:leftRightHook}
  $\rup$ and $\lup$ have the following properties.
  \begin{enumerate}
  \item
    $\up s= \rup s\wedge \lup s$.
  \item
    If $s\lexl t$ in $2^{\ast}$ then $\rup t\leq \rup s$;
    if $s\lexu t$ then $\lup s\leq \lup t$.
  \item
    $\rup(s\msign)= \rup s$; $\lup(s\psign)= \lup s$.
  \item
    $\rup s\vee \lup s=2^{\omega}$.
  \item
    If $t<s$ then $\rup s\wedge \lup t=\emptyset$.
  \item
    $\up s\leq \rup(s\psign)\vee \lup(s\msign)$.
  \end{enumerate}
\end{lemma}
\begin{proof}
  (1) Suppose $t$ and $u$ are right and left bristles of $s$. They both differ
  from $s$, but cannot differ at the same place. Thus they must differ from each
  other, and we deduce that $\up t\wedge \up u=\emptyset$.

  (2) We prove only the first assertion, since the second is dual. If
  $s\sqsubseteq t$ then $\up t\leq \up s$, and any right bristle of
  $t$ either is a right bristle of $s$ or has $s$ as a prefix. If $s<t$ then
  there is a unique $t'\sqsubseteq t$ such that $t'$ is a right
  bristle of $s$. Then $\up t\leq \up t'$. Also, any right
  bristle of $t$ either is a right bristle of $t'$ -- and hence of $s$
  -- or has $t'$ as a prefix.

  (3) From $s\lexl s\msign$ we deduce $\rup(s\msign)\leq \rup
  s$. For the reverse, any right bristle of $s$ is also a right bristle of $s\msign$.
  Also, $\up s= \up(s\msign)\vee \up(s\psign)$, and $s\psign$ is a right
  bristle of $s\msign$. The other assertion is dual.

  (4) 
  We use induction on the length of $s$; the base case $s=\varepsilon$ is obvious.
  Using part (3), and also the fact that $s$ and $s\msign$ have the same left bristles, we find that
  \[
    \rup(s\msign)\vee\lup(s\msign)
      = \rup s \vee \up(s\msign) \vee \bigvee\{\up t \mid t \text{ a left bristle of } s \}
      = \rup s \vee \lup s = 2^\omega
    \text{.}
  \]
  By symmetry the same works for $s\psign$.

  (5) Let $u$ be the greatest common prefix of $s$ and $t$:
  then $u\msign \sqsubseteq t$ and $u\psign \sqsubseteq s$.
  It suffices to consider the case for $\lup(u\msign)\wedge\rup(u\psign)$,
  which is the meet of
  \[
    \left( \up(u\msign)\vee\bigvee\{\up u'\mid u' \text{ a left bristle of } u \right)
  \]
  and
  \[
    \left( \up(u\psign)\vee\bigvee\{\up u''\mid u'' \text{ a right bristle of } u \right)
    \text{.}
  \]
  If $u'$ and $u''$ are bristles as described,
  then $u\msign < u\psign$, $u\msign < u''$, $u' < u\psign$ and $u' < u < u''$
  and it follows that all the meets got by redistributing the expression are 0.

  (6) Because $\up s= \up(s\msign)\vee \up(s\psign)$.
\end{proof}

\begin{lemma}\label{lem:CantorPresn2}
\begin{align*}
\Omega 2^{\omega} \cong \Fr\langle \rup s, \lup s \text{ (}s\in2^{\ast}\text{)}\mid
      \rup t    &  \leq \rup s\text{ (}s\lexl t\text{)}\\
      \rup s    &  \leq \rup(s\msign)\\
      \lup s    &  \leq \lup t\text{ (}s\lexu t\text{)}\\
      \lup s    &  \leq \lup(s\psign)\\
      \true     &  \leq \rup\varepsilon\\
      \true     &  \leq \lup\varepsilon\\
      \true     &  \leq \rup s\vee \lup s\\
      \rup s\wedge \lup t  &  \leq \false  \text{ (}t<s\text{)}\\
      \rup s\wedge \lup s  &  \leq \rup(s\psign)\vee \lup(s\msign)
    \rangle
\end{align*}
\end{lemma}

\begin{proof}
The homomorphism from the frame as presented here to $\Omega2^{\omega}$ takes
$\rup s$ and $\lup s$ to the opens as in
Definition~\ref{def:leftRightHook}, and then Lemma~\ref{lem:leftRightHook}
shows that the relations are respected.
In the other direction we map $\up s$ to $\rup s\wedge \lup s$ and it is easily
shown that all the relations are respected.
In particular, for respect of the relation $\up s= \up(s\msign)\vee \up(s\psign)$ we must have
\begin{equation}\label{eq:upToLRup}
  \rup s\wedge \lup s= \left( \rup(s\msign)\wedge \lup(s\msign)\right)
    \vee  \left(  \rup(s\psign)\wedge \lup(s\psign) \right)
  \text{.}
\end{equation}
For $\geq$ we use that $\lup(s\pm)\leq\lup s$ and similarly for $\rup$.
For $\leq$ we apply distributivity to the right hand side.
For three of the conjuncts we use $\rup s \leq \rup(s\msign)$ and $\lup s \leq \lup(s\psign)$;
for the other we use the final relation
$\rup s\wedge \lup s  \leq \rup(s\psign)\vee \lup(s\msign)$.

Now Lemma~\ref{lem:leftRightHook}~(1) shows that one composite
takes $\up s$ to $\rup s\wedge \lup s$ and then back to $\up s$, so is the identity.
To show the other composite is the identity we need%
\[
  \rup s=(\rup s\wedge \lup s)\vee\bigvee_{t\in\mathord{\mathrm{RB}}(s)}
    (\rup t\wedge \lup t)\text{,}%
\]
where $\mathord{\mathrm{RB}}(s)$ is the set of right bristles for $s$, and
similarly for $\lup s$. The $\geq$ direction is easy, since if $t$
is a right bristle of $s$ then $s\lexl t$ and so $\rup
t\leq \rup s$.

For $\leq$ we use induction. The base case, $s=\varepsilon$, is clear. For the
induction step,%
\begin{align*}
\rup(s\pm)  &  = \rup(s\pm)\wedge \rup
s= \rup(s\pm)\wedge\left(  (\rup s\wedge
 \lup s)\vee\bigvee_{t\in\operatorname*{RB}(s)}(\rup
t\wedge \lup t)\right) \\
&  \leq(\rup(s\pm)\wedge \rup s\wedge
  \lup s)\vee\bigvee_{t\in\operatorname*{RB}(s\pm)}(\rup t\wedge
 \lup t)
\end{align*}
since every right bristle of $s$ is also a right bristle of $s\pm$.
Now using equation~\eqref{eq:upToLRup} we have
\[
  \rup(s\msign)\wedge \rup s\wedge \lup s\leq\left(  \rup(s\msign)\wedge \lup(s\msign)\right)  \vee
  \bigvee_{t\in\operatorname*{RB}(s\msign)}(\rup t\wedge \lup t)
\]
since $s\psign$ is a right bristle of $s\msign$, and%
\[
\rup(s\psign)\wedge \rup s\wedge \lup s\leq \rup(s\psign)\wedge \lup(s\psign)
\]
since $s\msign <s\psign$ giving $\rup(s\psign)\wedge \lup(s\msign)\leq0$.
\end{proof}

\section{$\interval$ is a convex body}\label{sec:ConvBody}
The main task in this section is to prove that $\interval$, as midpoint algebra,
is iterative.
We shall use the fact that it can be described as a localic completion~\cite{LocCompA},
and then to construct the map $M$ as in Definition~\ref{def:itMidptAlg} we shall use
approximations in the ball domain (\cite{LocCompB}, following the ideas of \cite{EdHeck}).

Recall that for the localic completion of a generalized metric space $X$ we use
the elements $(x,\varepsilon)\in X\times Q_+$,
where $Q_{+}$ is the set of positive rationals,
as ``formal open balls'' $B_\varepsilon(x)$ (centre $x$, radius $\varepsilon$).
We write $\ball(X)$ for $X\times Q_+$
and equip it with a transitive, interpolative ``refinement'' order%
\[
(x,\delta)\subset(y,\varepsilon)\text{ if }X(y,x)+\delta<\varepsilon\text{.}%
\]
Then the \emph{ball domain} $\Ball(X)$ is defined to be
the continuous dcpo \mbox{$\Idl(\ball(X),\supset)$}
(see~\cite{Infosys}).
Note that the small balls, the refined ones, are high in the
order. We therefore think of the points of the ball domain as rounded
\emph{filters} of formal balls.

There is a \emph{radius map} $r\colon \Ball(X)\rightarrow
\overleftarrow{[0,\infty)}$, with $r(F)$ the inf of the radii of the formal
balls in $F$. ($\overleftarrow{[0,\infty)}$ is the locale whose points are the
upper reals in that interval, namely inhabited, rounded, up-closed sets of
positive rationals.)

The localic completion $\overline{X}$ embeds in $\Ball(X)$; its
points are the \emph{Cauchy} filters, those containing formal balls of
arbitrarily small radius, i.e. the points of $\Ball(X)$ with
radius $0$.

\begin{proposition}\label{prop:Imetric}
  $\interval$ is the localic completion of the metric space $D$,
  the set of dyadic rationals (those with denominator a power of 2) in the range $(-1,1)$,
  with the usual metric.
\end{proposition}
\begin{proof}
  In~\cite{LocCompA} it is shown that $\mathbb{R}$ is the localic completion of $\mathbb{Q}$.
  We have to deal with two differences.
  First, $\mathbb{Q}$ is replaced by the dyadics, which is essentially straightforward
  because the dyadics are dense in the rationals.
  Note that although the centre $q$ of a formal ball must now be dyadic,
  the radius $\delta$ can be any positive rational.
  Second, we restrict to the closed interval.
  For a Dedekind section $S=(L,U)$ that is equivalent to imposing the geometric axioms
  $1\notin L$ and $-1\notin U$.

  The proof in~\cite{LocCompA} sets up a geometric bijection between Dedekind sections $S$
  and Cauchy filters $F$ of $\mathbb{Q}$ as follows.
  The Dedekind section $S(F)$ has for its upper and lower sections the two sets
  $\{ q\pm\delta \mid (q,\delta)\in F \}$.
  The Cauchy filter $F(S)$ comprises those $(q,\delta)$ for which
  $q-\delta < S < q+\delta$,
  where of course we now have to restrict to $q\in D$.

  The main difficulty is in showing that $S\subseteq S(F(S))$.
  Suppose $q<S$. We can find dyadic $q'$ with $q<q'<S$,
  and so without loss of generality we can assume $q$ is dyadic.
  We know that $q<1$ (otherwise $1<S$).
  Let $r=\frac14(q + 3)$, which is dyadic, and $\delta=\frac34(1-q)$.
  Then $q = r-\delta$, $r<1$ and $r+\delta = 1+\frac12(1-q)>1$
  so $S<r+\delta$.
  If $r\in D$ then $(r,\delta)$ provides a ball to show $q<S(F(S))$.
  If $r\leq -1$ (so also $q<-1$) then instead we can use $(0,-q)$.
  The argument for $S<q$ is symmetric.

  We also show that $F(S(F))\subseteq F$.
  Suppose $(r,\varepsilon),(r',\varepsilon')\in F$,
  so that $r-\varepsilon < S(F) < r'+\varepsilon'$.
  This interval is the ball $(q,\delta)$ where
  $q=\frac12(r-\varepsilon+r'+\varepsilon')$ and
  $\delta=\frac12(r'+\varepsilon'-r+\varepsilon)$.
  We must show that if $q\in D$ then $(q,\delta)\in F$,
  but this is so because there is some common refinement in $F$
  of $(r,\varepsilon)$ and $(r',\varepsilon')$,
  and it also refines $(q,\delta)$.
\end{proof}

We extend the midpoint map $m\colon \interval\times\interval\rightarrow\interval$
by allowing the second argument to be taken from a ball domain. In
$\Ball(D)$ we have a point with centre $0$ and radius $1$. As a
filter, it comprises those formal balls $(q,\delta)\supset(0,1)$. Let $B$ be
the up closure in $\Ball(D)$ of this point, and write $\bot$
for the point since it is bottom in $B$.
Note that if $F\in\Ball(D)$, then $\bot\sqsubseteq F$ iff $(0,1+\varepsilon)\in F$
for all $\varepsilon\in Q_+$.

\begin{lemma}
The embedding $i\colon \interval\hookrightarrow\Ball(D)$ factors via
$B$.
\end{lemma}
\begin{proof}
If $\varepsilon>0$ then we can find $r\in D$ with $(r,\varepsilon/2)\in x$.
Then $(0,1+\varepsilon)\supset(r,\varepsilon/2)$ and so is in $x$.
\end{proof}

We define $m'\colon \interval\times B\rightarrow B$ as follows.
Let $x$ and $F$ be in $\Ball(D)$ with $x$ Cauchy and $F\supseteq\bot$.
We define%
\[
m'(x,F)= \mathord{\supset}\{(m(q,r),m(\delta,\varepsilon))\mid(q,\delta)\in
x,(r,\varepsilon)\in F\}
\]
(i.e. the set of all formal balls refined by one in the set on the right).
The fact that it is a filter follows from the fact that if $(q,\delta)\supset(q',\delta')$ in $x$
and $(r,\varepsilon)\supset(r',\varepsilon')$ in $F$ then%
\[
  (m(q,r),m(\delta,\varepsilon))\supset(m(q',r'),m(\delta',\varepsilon'))
  \text{.}
\]
This is because
\[
  \left|  \frac{q+r}{2}-\frac{q'+r'}{2}\right|
    +\frac{\delta'+\varepsilon'}{2}
  \leq\frac{1}{2}\left(  |q-q'|+\delta'+|r-r'|+\varepsilon'\right)
  \leq \frac{\delta+\varepsilon}{2}
  \text{.}%
\]

To see that it is bigger than $\bot$, suppose $\varepsilon>0$.
Since $x$ is Cauchy, there is some $(q,\delta)\in x$ with $\delta<\varepsilon/2$;
also, $(0,1+\varepsilon/2)\in F$ and so
$(q/2,\frac{1}{2}+\frac{\delta}{2} +\frac{\varepsilon}{4})\in m'(x,F)$.
From $|q|\leq1$ it follows that
$(0,1+\varepsilon)\supset(q/2,\frac{1}{2}+\frac{\delta}{2}+\frac{\varepsilon }{4})$
and so $(0,1+\varepsilon)\in m'(x,F)$.

\begin{lemma}
  \begin{enumerate}
  \item
    $m=m'\circ(\interval\times i)$.
  \item
    $r\circ m'(x,F)=r(F)/2$.
  \end{enumerate}
\end{lemma}
\begin{proof}
  Both are clear.
\end{proof}

\begin{theorem}\label{thm:IIterativeMPA}
  The midpoint algebra $\interval$ is iterative.
\end{theorem}

\begin{proof}
Let $X$ be a locale and $h\colon X\rightarrow\interval$, $t\colon X\rightarrow X$ be two
maps.
We require a unique morphism $M\colon X\rightarrow\interval$ making the following diagram commute.%
\[%
  \xymatrix{
    {\interval\times X}
      \ar@{->}[r]^{\interval\times M}
    & {\interval\times\interval}
      \ar@{->}[d]^{m}
    \\
    {X}
      \ar@{->}[r]_{M}
      \ar@{->}[u]^{\langle h,t\rangle}
    & {\interval}
  }
\]

$\Loc(X,B)$ is a dcpo with bottom.
We define a Scott continuous endofunction $T$ on it
by $T(f)=m'\circ(\interval\times f)\circ\langle h,t\rangle$:%
\[%
  \xymatrix{
    {\interval\times X}
      \ar@{->}[r]^{\interval\times f}
    & {\interval\times B}
      \ar@{->}[d]^{m'}
    \\
    {X}
      \ar@{->}[r]_{T(f)}
      \ar@{->}[u]^{\langle h,t\rangle}
    & {B}
  }
\]

Let $M$ be its least fixpoint, $\bigdsqcup_{n} M_{n}$ where $M_{0}$
is constant $\bot$ and $M_{n+1}=T(M_{n})$. Then $r\circ M=\frac{1}{2}(r\circ
M)$, from which it follows that $r\circ M=0$ and $M$ factors via $\interval$
thus giving us existence of the required $M$.

For uniqueness, suppose $M'$ is another such. Then $M\sqsubseteq
M'$ since $M$ is least fixpoint, but the specialization order on
$\interval$ is discrete.
\end{proof}

We can calculate the inverse image function for $M$ in the above theorem more
explicitly, at least for the subbasic opens $(p,\alpha)$. First of all,%
\[
  M_{0}^{\ast}(p,\alpha)=
    \left\{
      \begin{array}[c]{ll}%
        \top & \text{if }(p,\alpha)\supset(0,1)\\
        \bot & \text{otherwise}%
      \end{array}
    \right.
\]
(and note that the condition is decidable). Next,%
\[
T(f)^{\ast}(p,\alpha)=\bigvee\{h^{\ast}(q,\delta)\wedge t^{\ast}f^{\ast
}(r,\varepsilon)\mid(p,\alpha)\supset(\frac{q+r}{2},\frac{\delta+\varepsilon
}{2})\}\text{.}%
\]
This allows us to calculate $M^{\ast}(p,\alpha)=\bigdvee_{n} M_{n}^{\ast}(p,\alpha)$.

\section{The map $c\colon 2^{\omega}\rightarrow\interval$}\label{sec:c}

Thinking of the signs in a point of Cantor space $2^{\omega}$ as standing for
$1$ or $-1$, such an infinite sequence can be viewed as a binary expansion,
thus giving a map to $\interval$.

\begin{definition}
We define a map $c\colon 2^{\omega}\rightarrow\interval$ as $M_{-1,+1}$.
It is characterized by the equation%
\[
c(\pm s)=\frac{1}{2}\left(  \pm1+c(s)\right)  \text{.}%
\]
\end{definition}

From the characterizing equation we see that, in more traditional form,%
\begin{equation}\label{eq:c}
  c((s_{i})_{i=1}^{\infty})=\sum_{i=1}^{\infty}\frac{s_{i}}{2^{i}}\text{.}%
\end{equation}

\begin{definition}
$2^{\ast}$ is the discrete space of finite sequences of signs. We define
$c'\colon 2^{\ast}\rightarrow\interval$ by the formula~\eqref{eq:c},
adapted for finite sequences. Thus we think of the finite sequence $s$ as the
infinite sequence $s0^{\omega}$ (which is not in $2^{\omega}$, of course).
\end{definition}

$c'$ is an isomorphism between $2^{\ast}$ and $D$.

If $s$ is finite of length $n$ and $t$ is infinite, then we see from the
definition that $c(st)=c'(s)+2^{-n}c(t)$.

We now show how to calculate the inverse image function $c^{\ast}$,
using Theorem~\ref{thm:IIterativeMPA} and the remarks following it.
Our map $h\colon 2^{\omega}\rightarrow\interval$ is $h(\pm s)= \pm 1$.
It has%
\[
  h^{\ast}(p,\alpha)=\left\{
    \begin{array}[c]{ll}%
      \up \psign & \text{if }p-\alpha<1<p+\alpha\\
      \emptyset & \text{otherwise}%
    \end{array}
    \right\}  \vee \left\{
      \begin{array}[c]{ll}%
        \up\msign & \text{if }p-\alpha<-1<p+\alpha\\
        \emptyset & \text{otherwise}%
      \end{array}
    \right\}
    \text{.}%
\]
Hence, for $f\colon 2^{\omega}\rightarrow\interval$,%
\begin{align*}
T(f)^{\ast}(p,\alpha)  &  =\bigvee\{(\up \psign)\wedge t^{\ast}f^{\ast
}(r,\varepsilon)\mid(p,\alpha)\supset(\frac{q+r}{2},\frac{\delta+\varepsilon
}{2}),q-\delta<1<q+\delta\}\\
&  \vee\bigvee\{(\up \msign)\wedge t^{\ast}f^{\ast}(r,\varepsilon)\mid
(p,\alpha)\supset(\frac{q+r}{2},\frac{\delta+\varepsilon}{2}),q-\delta
<-1<q+\delta\}\text{.}%
\end{align*}

\begin{lemma}
  In $\Omega\mathbb{R}$ we have%
  \begin{align*}
    \bigvee\{(r,\varepsilon) \mid(p,\alpha)\supset(\frac{q+r}{2},\frac{\delta+\varepsilon}{2}),
        q-\delta<-1<q+\delta\} & =(2p+1,2\alpha)\text{,}\\
    \bigvee\{(r,\varepsilon) \mid(p,\alpha)\supset(\frac{q+r}{2},\frac{\delta+\varepsilon}{2}),
        q-\delta<1<q+\delta\} & =(2p-1,2\alpha)  \text{.}
  \end{align*}
\end{lemma}
\begin{proof}
  We prove only the first, since the second follows by symmetry.
  We have
  \begin{align*}
    (r,\varepsilon) \subset (2p+1,2\alpha)
      & \Leftrightarrow \left( \frac{-1+r}{2}, \frac{\varepsilon}{2}\right) \subset (p,\alpha) \\
      & \Leftrightarrow \exists\beta>0
          \left( \frac{-1+r}{2}, \beta+\frac{\varepsilon}{2}\right) \subset (p,\alpha)
  \end{align*}
  Then the final condition is equivalent to the existence of $q,\delta$,
  with $-1<q<-1+\delta$ and
  \[
    \left(\frac{q+r}{2},\frac{\delta+\varepsilon}{2}\right) \subset (p, \alpha)
    \text{.}
  \]
  (Note that the second condition is equivalent to this with $q=-1,\delta=0$,
  and the $\beta$ enables us to fatten $-1$ out to a positive ball.)
  Each $\left(\frac{q+r}{2},\frac{\delta+\varepsilon}{2}\right)$
  can be refined to a $\left( \frac{-1+r}{2}, \beta+\frac{\varepsilon}{2}\right)$
  and vice versa.
\end{proof}

In $\Omega\interval$ the same equations hold,
but we must be careful how we interpret the right-hand side.
Consider the first equation.
If $p<0$ then the centre $2p+1$ of the ball on the right is still in $D$.
The ball is approximated from below by refinements with the same centre,
and it follows in the proof that we can restrict the balls appearing in the left-hand side
to those with centre in $D$.

Now suppose $0\leq p$, so that $1\leq 2p+1$.
Then the ball $(2p+1,2\alpha)$ is equivalent in $\Omega\interval$ to the interval
$(2p+1-2\alpha,1]$.
This interval may take various forms depending on the value of $2p+1-2\alpha$
-- which, in particular, may be less than $-1$ or greater than $1$.
However, in every case it is approximated by balls refining $(2p+1,2\alpha)$
and with centre in $D$.
Therefore the equations in the lemma will still hold in $\Omega\interval$.

Taking care with interpretations in $\Omega\interval$ in that way, it follows that%
\[
  T(f)^{\ast}(p,\alpha)
    =(\up \psign)\wedge t^{\ast}f^{\ast}(2p-1,2\alpha)
      \vee(\up \msign)\wedge t^{\ast}f^{\ast}(2p+1,2\alpha)\text{.}%
\]
Although our proof of iterativity used the metric space structure and the opens balls,
we shall be actually be more interested in the behaviour of the half-open intervals.
In the rest of the section we shall calculate formulae for opens such as
$c^\ast((c'(s),1])$.
First, rewriting $p-\alpha$ as $p$, we see, for all $p$, that
\begin{equation}\label{eq:TfHalfOpens}
  T(f)^{\ast}(p,1]
    =(\up \psign)\wedge t^{\ast}f^{\ast}(2p-1,1]
      \vee(\up \msign)\wedge t^{\ast}f^{\ast}(2p+1,1]
  \text{.}%
\end{equation}

Now if $p=c'(s)\in D$, we have%
\begin{align*}
  (2p-1,1]  &  =\left\{
    \begin{array}[c]{ll}%
      (c'(s'),1] & \text{if }s=\psign s'\\
      (-1,1]=\bigdvee_{k}(c'(\msign^{k}),1] & \text{if }s=\varepsilon\\
      \interval & \text{if }s=\msign s'%
    \end{array}
  \right. \\
  (2p+1,1]  &  =\left\{
    \begin{array}[c]{ll}%
      \emptyset & \text{if }s=\psign s'\text{ or }s=\varepsilon\\
      (c'(s'),1] & \text{if }s=\msign s'%
    \end{array}
  \right.
\end{align*}

Using this we can calculate $c^{\ast}(c'(s),1]$ by induction on the
length of $s$, the base case requiring knowledge of $c^{\ast}(-1,1]$.

\begin{lemma}\label{lem:cStar}
  \begin{enumerate}
  \item
    $c^{\ast}(c'(\msign^{k}),1]
      =\bigvee_{i=0}^{k-1}\up(\msign^{i}\psign)\vee((\up\msign^{k})\wedge (t^\ast)^k c^{\ast}((0,1]))$.
  \item
    $c^{\ast}(-1,1]=\bigvee_{i=0}^{\infty}\up(\msign^{i}\psign)$.
  \item
    $c^{\ast}(0,1]=\bigvee_{i=0}^{\infty}\up(\psign\msign^{i}\psign)$.
  \end{enumerate}
\end{lemma}

\begin{proof}
  (1) is by induction on $k$. The base case, $k=0$, is clear.
  \begin{align*}
    c^{\ast}(c'(\msign^{k+1}),1]
      & = (\up \psign)\wedge t^{\ast}c^{\ast}(\interval)
          \vee(\up \msign)\wedge t^{\ast}c^{\ast}(c'(\msign^k,1])
                                 \quad \text{(equation~\eqref{eq:TfHalfOpens})} \\
      & = (\up \psign) \vee (\up \msign)\wedge t^{\ast}\left(
        \bigvee_{i=0}^{k-1}\up(\msign^{i}\psign)\vee((\up\msign^{k})\wedge (t^\ast)^k c^{\ast}((0,1]))
        \right) \\
    &  =\bigvee_{i=0}^{k}\up(\msign^{i}\psign)\vee((\up\msign^{k+1})\wedge (t^\ast)^{k+1} c^{\ast}((0,1]))
  \end{align*}

  (2) Using part~(1), and applying equation~\eqref{eq:TfHalfOpens} to $c^\ast(0,1]$, we see that
  \begin{align*}
    c^{\ast}(c'(\msign^{k}),1]
      &  =\bigvee_{i=0}^{k-1}\up(\msign^{i}\psign)
          \vee((\up\msign^{k})\wedge (t^\ast)^k((\up\psign)\wedge t^{\ast}c^{\ast}((-1,1])))\\
      &  =\bigvee_{i=0}^{k-1}\up(\msign^{i}\psign)
          \vee((\up\msign^{k}\psign)\wedge (t^\ast)^{k+1}c^{\ast}((-1,1]))\\
      &  \leq\bigvee_{i=0}^{k}\up(\msign^{i}\psign)\leq c^{\ast}(c'(\msign^{k+1}),1]\text{.}%
  \end{align*}
  It follows that%
  \[
    c^{\ast}(-1,1]
      =c^{\ast}\left(  \bigdvee_{k}(c'(\msign^{k}),1]\right)
      =\bigdvee_{k}\bigvee_{i=0}^{k}\up(\msign^{i}\psign)
      =\bigvee_{i=0}^{\infty}\up(\msign^{i}\psign)\text{.}%
  \]

  (3) Apply equation~\eqref{eq:TfHalfOpens} with $p=0$, and then use part~(2).
\end{proof}

In other words, $c(u)>-1$ iff $u$ has a $\psign$ somewhere; and $c(u)>0$ iff $u$
starts with a $\psign$ and has at least one more.

\begin{proposition}\label{prop:cStar}
  If $s\in2^{\ast}$ then
  \begin{enumerate}
  \item
    $c^{\ast}((c'(s),1])=\bigdvee_{k}\rup(s\psign\msign^{k}\psign)$, and
  \item
    $c^{\ast}([-1,c'(s)))=\bigdvee_{k}\lup (s\msign\psign^{k}\msign)$.
  \end{enumerate}
\end{proposition}

\begin{proof}
We prove only the first assertion, since the second is dual. We use induction
on the length of $s$.

For $s=\varepsilon$, we use Lemma~\ref{lem:cStar}~(3) together with
$\rup(\psign\msign^{k}\psign)=\bigvee_{i=0}^{k}\up(\psign\msign^{i}\psign)$. Now we can use
the previous calculations and see%
\begin{align*}
  c^{\ast}((c'(\psign s),1])
    &  =(\up\psign)\wedge t^{\ast}c^{\ast}((c'(s),1])\\
    &  =(\up\psign)\wedge t^{\ast}\left(  \bigdvee_{k}\rup(s\psign\msign^{k}\psign)\right) \\
    &  =\bigdvee_{k}\rup(\psign s\psign\msign^{k}\psign)\\
  c^{\ast}((c'(\msign s),1])
    &  =(\up\psign)\wedge t^{\ast}2^\omega \vee (\up\msign)\wedge t^{\ast}c^{\ast}((c'(s),1])\\
    &  =(\up\psign)\vee(\up\msign)\wedge t^{\ast}
        \left(  \bigdvee_{k}\rup(s\psign\msign^{k}\psign)\right) \\
    &  =\bigdvee_{k}\rup(\msign s\psign\msign^{k}\psign)\text{.}%
\end{align*}
\end{proof}

\section{$c$ is a proper surjection}\label{sec:ProperSurjn}

We shall show that $c$ is a proper surjection in the sense of Vermeulen~\cite{Vermeulen:ProperML}:
the right adjoint $\forall_{c}\colon \Omega2^{\omega}\rightarrow\Omega\interval$
of $c^{\ast}$ preserves directed joins and satisfies a Frobenius condition.
$\forall_c$ is thus a preframe homomorphism.
We first use the preframe coverage theorem to present $\Omega2^{\omega}$ as a preframe,
and define $\forall_{c}$ by its action on a preframe base,
and then we show that this function is right adjoint to $c^\ast$ and has the Frobenius condition.

Any open of $2^{\omega}$ is a directed join of finite joins of basic opens
$\up s= \rup s\wedge\lup s$, hence a directed join
of finite meets of finite joins of opens of the form $\rup s$ and
$\lup s$. But since $\lexl$ and $\lexu$ are total orders,
by Lemma~\ref{lem:leftRightHook} we get a preframe base from opens of the
form $\rup s$, $\lup s$ or $\rup
s\vee \lup t$. Our strategy now is to calculate $\forall_{c}$ for
these and to rely on preservation of finite limits and directed joins to get
the rest.

\begin{definition}
The distributive lattice $S_{\rup}$ is defined as $2^{\ast}%
\cup\{\bot\}$, with $2^{\ast}$ ordered by the reverse of $\lexl$ and
$\bot$ an adjoined bottom. Since it is totally ordered it has binary meets and
joins, and also top $\varepsilon$ and bottom $\bot$.

Similarly we define $S_{\lup}=2^{\ast}\cup\{\bot\}$, with $2^{\ast}$
ordered by $\lexu$.

We write $S$ for $S_{\lup}\times S_{\rup}$.
\end{definition}

\begin{lemma}\label{lem:CantorPresn3}
\begin{align*}
\Omega2^{\omega}\cong\Fr\langle S & \text{ (qua }\vee \text{-semilattice)} \mid\\
  (s,t) & \leq (s,t\msign) \text{ (}t\in2^{\ast}\text{)}\\
  (s,t) & \leq (s\psign,t) \text{ (}s\in2^{\ast}\text{)}\\
  \true & \leq (s,\varepsilon)\text{ (}s\in S_{\lup}\text{)}\\
  \true & \leq (\varepsilon,s)\text{ (}s\in S_{\rup}\text{)}\\
  \true & \leq (s,t)\text{ (}s,t\in2^{\ast},t\lexu s\text{ or }t\lexl s\text{)}\\
(u,s)\wedge(t,v)  &  \leq (u,v)\text{ (if }t<s\text{ in }2^{\ast}\text{ and }(u,v)\leq(t,s)\text{)}\\
(u,s)\wedge(s,v)  &  \leq (s\msign,s\psign)\text{ (if }s\in2^{\ast}
    \text{ and }(u,v)\leq(s\msign,s\psign)\text{)}\rangle
\end{align*}
and%
\[
\Omega2^{\omega}\cong\PreFr\langle S\text{ (qua poset)}%
\mid\text{... same relations as above ...}\rangle
\]
\end{lemma}

\begin{proof}
To map from the presentation of Lemma~\ref{lem:CantorPresn2} to this one we
map $\rup s$ and $\lup s$ to $(\bot,s)$ and $(s,\bot)$.
This respects all the relations and so gives a frame homomorphism.
For the inverse we map $(\bot,\bot)$ to $\false$;
$(\bot,s)$ and $(s,\bot)$ to $\rup s$ and $\lup s$;
and $(s,t)$ to $\lup s\vee \rup t$.
Again this respects the relations and so gives a
frame homomorphism. As can be tested on generators, the two composites are
both identities.

The final part is now an application of the preframe coverage theorem
\cite{PrePrePre}, once it is checked that the relations are all join-stable.
This is straightforward.
Note the role of the condition $(u,v)\leq(t,s)$ in the last relation but one
(and similarly in the last).
For all $u,v$ we have $(u,s\vee v)\wedge (t\vee u,v) \leq (u,v)$,
and this is the form that naturally arises from join-stability.
However, if $t\leq u$ or $s\leq v$ then one of the two conjuncts is $(u,v)$
and the relation holds automatically in the preframe presented.
For the relations given we only need to consider the case where
$u\leq t$ and $v\leq s$.
\end{proof}

Our strategy now is to calculate $\forall_{c}$ for the opens $(s,t)$ and to
rely on preservation of finite meets and directed joins to get the rest.
Using Definition~\ref{def:theta} we define a preframe homomorphism that we
subsequently show to be $\forall_{c}$.
Let us explain roughly how the definition arises.
(We don't need a rigorous definition yet,
since the definition is checked in Theorem~\ref{thm:PropSurjn}.)
First consider $\forall_{c}(\bot,s)$, the biggest open
$U\in\Omega\interval$ such that $c^{\ast}U\leq \rup s$. If
$c(t)<c(u)$ then $t<u$ (it is much more complicated for $\leq$), and it
follows that if $c(s\msign^{\omega})<c(u)$ then $u$ is in $\rup s$.
Hence $(c(s\msign^{\omega}),1]\leq\forall_{c}(\bot,s)$. If $s$ contains a $\psign$ then
$\forall_{c}(\bot,s)$ cannot be any bigger, for it would then contain
$c(s\msign^{\omega})$ itself. By looking at the last $\psign$ in $s$ we can replace
$\psign\msign^{\omega}$ by $\msign\psign^{\omega}$ and find a $u$ in $c^{\ast}(\forall_{c}%
(\bot,s))$ but not in $\rup s$. Hence $\forall_{c}(\bot
,s)=(c(s\msign^{\omega}),1]$. If $s$ has no $\psign$ then the argument is slightly
different. $\rup s= \rup\varepsilon=2^{\omega}$, so we
know $\forall_{c}(\bot,s)=\interval$. Similarly, $\forall_{c}(s,\bot)$ is
either $[-1,c(s\psign^{\omega}))$ or $\interval$.

There remains $\forall_{c}(s,t)$. If $c(s\psign^{\omega})<c(t\msign^{\omega})$ then this
turns out to be $[-1,c(s\psign^{\omega}))\vee(c(t\msign^{\omega}),1]$ as one might
expect, while if $c(s\psign^{\omega})>c(t\msign^{\omega})$ it is $\interval$. However
we have to take some care where there is equality, since we then find that
$\lup s\vee \rup t$ is $2^{\omega}$ and so $\forall
_{c}(s,t)$ must be $\interval$ -- this is an instance where $\forall_{c}$
does not preserve finite joins.

\begin{definition}
If $s,t\in2^{\ast}$ we write $s\overlap t$ if (i) $t<s$, or (ii) $t\sqsubseteq
s$, or (iii) $s\sqsubseteq t$, or (iv) $s$ and $t$ are of the forms $u\msign\psign^{k}$
and $u\psign\msign^{l}$ respectively.
\end{definition}

\begin{lemma}\label{lem:overlap}
  \begin{enumerate}
  \item
    $s\overlap t$ iff $\lup s\vee \rup t=2^{\omega}$.
  \item
    If $s\overlap t$ then $c(t\msign^{\omega})\leq c(s\psign^{\omega})$.
  \item
    $\overlap$ is up-closed in $S$.
  \end{enumerate}
\end{lemma}

\begin{proof}
(1) $\Rightarrow$: In cases (i) and (ii) of the definition we have $t\lexl s$,
so $2^{\omega}= \lup s\vee \rup s\leq
 \lup s\vee \rup t$, similarly in case (iii). In case
(iv), we have $\rup t= \rup(u\psign\msign^{l})= \rup
(u\psign)$ and similarly $\lup s= \lup(u\msign)$. Now%
\begin{align*}
  \true  &  \leq \left(  \lup(u\msign)\vee \rup(u\msign)\right)
            \wedge\left(  \lup(u\psign)\vee \rup(u\psign)\right) \\
         &  \leq \lup(u\msign)\vee \rup(u\psign)\vee\left(
            \rup(u\msign)\wedge \lup(u\psign)\right) \\
         &  = \lup(u\msign)\vee \rup(u\psign)\text{ because }u\msign <u\psign
  \text{.}
\end{align*}

$\Leftarrow$: $\overlap$ is decidable.
Its negation is that $s<t$,
so that for some $u$ we have $u\msign \sqsubseteq s$ and $u\psign \sqsubseteq t$,
and in addition that either
$u\msign\psign^{k}\msign \sqsubseteq s$ or $u\psign\msign^{k}\psign \sqsubseteq t$ for some $k$.
Suppose the former.
Then $s<u\msign\psign^{\omega}<t$, so $u\msign\psign^{\omega}$ is in neither $\lup s$ nor $\rup t$.

(2) In case (i): if $u\msign \sqsubseteq t$, $u\psign \sqsubseteq s$, then
$c(t\msign^{\omega})<c'(u)<c(s\psign^{\omega})$. In case (ii) (and (iii) is
dual), we have $t\msign^{k}<s\psign$ for some $k$, and can use (i). In case (iv),
$c(t\msign^{\omega})=c(u\psign\msign^{\omega})=c'(u)=c(u\msign\psign^{\omega})=c(s\psign^{\omega})$.

(3) Suppose $s\overlap t$.
We show that if $t'\lexl t$ then $s\overlap t'$.
By symmetry it also follows that if $s\lexu s'$ then $s'\overlap t$,
and the result will follow.
We examine the cases of $s\overlap t$.
First, if $t\lexl s$ then $t'\lexl s$.

Second, suppose $s\sqsubseteq t$.
If $t'\sqsubseteq t$ then $s$ and $t'$ are comparable under $\sqsubseteq$.
Otherwise $t'<t$ and so $t'\lexu s$.

Finally, suppose $s=u\msign\psign^k, t=u\psign\msign^l$.

If $t'\sqsubseteq t$ then either $t'\sqsubseteq u \sqsubseteq s$
or $u\psign\sqsubseteq t' \sqsubseteq t$ and either way we get $s\overlap t'$.

There remains the case $t'<t$.
We have either $t'<u$, so $t'< s$, or $u\msign\sqsubseteq t'$.
In this latter case consider whether $t'$ has any further $\msign$ after $u\msign$.
If it does then $t'\lexu s$;
if not then $s$ and $t'$ are comparable under $\sqsubseteq$.
\end{proof}

\begin{definition}\label{def:theta}
We define a lattice homomorphism $\theta_{\rup}\colon S_{\rup}\rightarrow\Omega\interval$ by%
\[
\theta_{\rup}(t)=\left\{
\begin{array}
[c]{ll}%
\interval & \text{if }t\in2^{\ast}\text{ and }t\text{ contains no }\psign\\
(c(t\msign^{\omega}),1] & \text{if }t\in2^{\ast}\text{ and }t\text{ contains at
least one }\psign\\
\emptyset & \text{if }t=\bot
\end{array}
\right.
\]

Similarly we define $\theta_{\lup}\colon S_{\lup}\rightarrow\Omega\interval$ with $\theta_{\lup}(s)=[-1,c(s\psign^{\omega}))$ when
$s$ contains a $\msign$.

The monotone function $\theta\colon S_{\lup}\times S_{\rup}\rightarrow\Omega\interval$ is defined by%
\[
\theta(s,t)=\left\{
\begin{array}
[c]{ll}%
\interval & \text{if }s,t\in2^{\ast}\text{ and }s\overlap t\\
\theta_{\lup}(s)\vee\theta_{\rup}(t) & \text{otherwise}%
\end{array}
\right.
\]
Note that if $t$ contains no $\psign$ or $s$ contains no $\msign$ then $s\overlap t$.
\end{definition}
That $\theta_{\rup}$ and $\theta_{\lup}$ are lattice homomorphisms
is simply to say that they are monotone and preserve top and bottom.
The monotonicity of $\theta$ then follows from that and from Lemma~\ref{lem:overlap}~(3).

\begin{lemma}
  We can define a preframe homomorphism
  $\forall_{c}\colon \Omega2^{\omega}\rightarrow\Omega\interval$ by $\forall_{c}(s,t)=\theta(s,t)$.
\end{lemma}

\begin{proof}
One should check that the relations in Lemma~\ref{lem:CantorPresn3} are
respected. Much of this is routine. We consider the last two in more detail.

For the last but one, suppose $t<s$ and $(u,v)\leq (t,s)$. First,
\[
  (\theta_{\lup}(u)\vee\theta_{\rup}(s)) \wedge (\theta_{\lup}(t)\vee\theta_{\rup}(v))
    \leq \theta_{\lup}(u)\vee\theta_{\rup}(v) \vee (\theta_{\lup}(t)\wedge\theta_{\rup}(s))
    = \theta_{\lup}(u)\vee\theta_{\rup}(v)
  \text{.}
\]
This is because, given $t<s$, $t$ and $s$ must contain $\msign$ and $\psign$ respectively, so
\[
  \theta_{\lup}(t)\wedge\theta_{\rup}(s)
    = [-1,c(t\psign^\omega)) \wedge (c(s\msign^\omega),1] = \emptyset
\]
because $c(t\psign^\omega) \leq c(s\msign^\omega)$.

We still need to examine the cases where $\theta$ takes the value $\interval$.
Suppose $u\overlap s$.
(The case $t\overlap v$ is by symmetry.)
We must show $\theta(t,v)\leq\theta(u,v)$.
If $s\lexl u$ or $s\lexu u$ then from $t<s$ we find $t\lexu u$.
Now suppose $u=w\msign\psign^k, s=w\psign\msign^l$.
Since $t<s$, one possibility is that $t<w$, so $t<u$.
In all the cases so far $t\lexu u$, so $\theta(t,v)\leq\theta(u,v)$.
The remaining possibility (from $t<s$) is that $w\msign\sqsubseteq t$.
Since $u\lexu t$ we must have $t\sqsubseteq u$,
so $t=w\msign\psign^{k'}$ with $k'\leq k$, and $\theta_{\lup}(t)=\theta_{\lup}(u)$.
It remains only to consider the case where, in addition, $t\overlap v$.
From $t<s\lexl v$ we deduce $t<v$, and so $t\overlap v$ falls into its final case:
hence $v=w\psign\msign^{l'}$ for some $l'$, so $u\overlap v$.

The final relation, $(u,s)\wedge(s,v)\leq(s\msign,s\psign)$,
is clear since $s\msign\overlap s\psign$.
\end{proof}

\begin{theorem}\label{thm:PropSurjn}
  $c\colon 2^{\omega}\rightarrow\interval$ is a proper
  surjection, with $\forall_{c}$ right adjoint to $c^{\ast}$.
\end{theorem}

\begin{proof}
There are three things to show.

First, $c^{\ast}\circ\forall_{c}\leq\operatorname*{Id}$. For $s\overlap t$,
Lemma~\ref{lem:overlap} tells us that $(s,t)=2^{\omega}$. For the other case
it remains to show that $c^{\ast}(\theta_{\rup}(t))\leq
~\rup t$ (and similarly for $\lup$). If $t$ has no $\psign$
then $\rup t=2^{\omega}$, and otherwise we have%
\begin{align*}
c^{\ast}(\theta_{\rup}(t))
  &  =c^{\ast}((c(t\msign^{\omega}),1])
     =\bigdvee_{k}c^{\ast}\left(  (c'(t\msign^{k}),1]\right) \\
  &  =\bigdvee_{kl}\rup(t\msign^{k}\psign\msign^{l}\psign)\leq \rup t\text{.}%
\end{align*}

Second, $\forall_{c}\circ c^{\ast}=\operatorname*{Id}$.
It suffices to check this for opens of the form $[-1,c'(s))$, $(c'(t),1]$ and
$[-1,c'(s))\vee(c'(t),1]$,
since they form a preframe base of $\interval$. We have%
\begin{align*}
\forall_{c}\circ c^{\ast}\left(  [-1,c'(s))\vee(c'(t),1]\right)
   &  =\bigdvee_{kl}\forall_{c}(s\msign\psign^{k}\msign,t\psign\msign^{l}\psign)\\
   &  \geq\bigdvee_{kl}\left(
        [-1,c(s\msign\psign^{k}\msign\psign^{\omega}))
            \vee(c(t\psign\msign^{l}\psign\msign^{\omega}),1]\right) \\
   &  =\bigdvee_{k}[-1,c'(s\msign\psign^{k}))\vee\bigdvee_{l}(c'(t\psign\msign^{l}),1]\\
   &  =[-1,c'(s))\vee(c'(t),1]\text{.}%
\end{align*}

We have equality provided we have no $s\msign\psign^{k}\msign \overlap t\psign\msign^{l}\psign$
(and also, by a similar calculation, for the opens $[-1,c'(s))$ and $(c'(t),1]$).
If $c'(t)<c'(s)$ then $[-1,c'(s))\vee(c'(t),1]=\interval$,
so it remains to prove that if $c'(s)\leq c'(t)$
then we have no $s\msign\psign^{k}\msign \overlap t\psign\msign^{l}\psign$.
That is to say, for all $k,l$ we have $s\msign\psign^{k}\msign <t\psign\msign^{l}\psign$
(so for some $u$ we have $u\msign \sqsubseteq s\msign\psign^{k}\msign$ and $u\psign \sqsubseteq t\psign\msign^{l}\psign)$,
and for some $m$ we have either $u\msign\psign^{m}\msign \sqsubseteq s\msign\psign^{k}\msign$ or
$u\psign\msign^{m}\psign \sqsubseteq t\psign\msign^{l}\psign$.
(See Lemma~\ref{lem:overlap}.)
From $c'(s)\leq c'(t)$ we get three cases.
If $s<t$ then $u$ is a common prefix of $s$ and $t$
and in fact we have $m$ with $u\msign\psign^{m}\msign \sqsubseteq s\msign$.
If $s\sqsubseteq t$ then from $c'(s)\leq c'(t)$ we cannot have $s\msign \sqsubseteq t$,
so we can take $u=s$ and either $s=t$ or $s\psign \sqsubseteq t$.
Either way, $u\psign \sqsubseteq t\psign$.
Then we can take $m=k$.
The argument for $t\sqsubseteq s$ is similar.

Third, the Frobenius condition $\forall_{c}(a\vee c^{\ast}b)=\forall_{c}a\vee
b$ -- in fact only the $\leq$ direction is necessary now. It suffices to check
the case where $a$ and $b$ are preframe basics. Suppose $a$ and $b$ are
$(s,t)\in S$ and $[-1,c'(s'))\vee(c'(t'),1]$,
so%
\[
  \forall_{c}(a\vee c^{\ast}b)
    =\bigdvee_{kl}\forall_{c}((s,t)\vee(s'\msign\psign^{k}\msign,t'\psign\msign^{l}\psign))\text{.}%
\]
We must therefore check
$\forall_{c}((s,t)\vee(s'\msign\psign^{k}\msign,t'\psign\msign^{l}\psign))\leq\forall_{c}a\vee b$
for each $k,l$.
Let $s''$ be
the greater of $s,s'\msign\psign^{k}\msign$ with respect to $\lexu$,
and let $t''$ be the smaller of $t,t'\psign\msign^{l}\psign$ with respect to $\lexl$.
Unless $s''\overlap t''$, we have
$\forall_{c}(s'',t'')=\theta_{\lup}(s'') \vee\theta_{\rup}(t'') \leq\forall_{c}a\vee b$.
(Note that
\[
  \theta_{\lup}(s'\msign\psign^k \msign) = [-1,c(s'\msign\psign^k \msign \psign^\omega))
    = [-1,c'(s'\msign\psign^k)) \leq [-1,c'(s'))
    \text{,}
\]
and similarly for $\theta_{\rup}(t'\psign\msign^{l}\psign$)$.)$
Also, if $s''$ and $t''$ are either $s$ and $t$ or
$s'\msign\psign^{k}\msign$ and $t'\psign\msign^{l}\psign$ then
$\forall_{c}(a\vee c^{\ast}b)\leq\forall_{c}a$
or $\forall_{c}(a\vee c^{\ast}b)\leq\forall_{c}c^{\ast}b=b$.

There are two remaining cases where we must consider $s''\overlap t''$,
but each follows from the other by $\psign$-$\msign$ duality,
so we consider $s\lexu s'\msign\psign^{k}\msign \overlap t\lexl t'\psign\msign^{l}\psign$.
From Lemma~\ref{lem:overlap} we see%
\[
  c(t\msign^{\omega})\leq c(s'\msign\psign^{k}\msign\psign^{\omega})=c'(s'\msign\psign^{k})<c'(s')
\]
so that%
\begin{align*}
  \forall_{c}((s,t)\vee(s'\msign\psign^{k}\msign,t'\psign\msign^{l}\psign))
    &  =\forall_{c}(s'\msign\psign^{k}\msign,t)=\interval\\
    &  =[-1,c'(s'))\vee(c(t\msign^{\omega}),1]\\
    &  \leq\forall_{c}(s,t)\vee\lbrack-1,c'(s'))\vee(c'(t'),1]=\forall_{c}a\vee b
    \text{.}%
\end{align*}

We have neglected the preframe basics where one of $s,t$ is $\bot$, or we just
have $[-1,c'(s'))$ or $(c'(t'),1]$. However,
these cases can easily be covered in the reasoning above.
\end{proof}

\section{$\interval$ as coequalizer of maps to Cantor space}\label{sec:Coequ}

We observe that $0_{\msign}=\psign\msign^{\omega}$ and $0_{\psign}=\msign\psign^{\omega}$ in $2^{\omega}$
are both mapped by $c$ to $0$. This is the starting point for describing $c$
as a coequalizer of two maps to $2^{\omega}$.

\begin{definition}
  We define two maps $u_{\pm}\colon 2^{\ast}\rightarrow2^{\omega}$
  by $u_{\pm}(s)=s0_{\pm}$.
\end{definition}

Since $c(0_{\msign})=c(0_{\psign})$, it is clear that $c\circ u_{\msign}=c\circ u_{\psign}$. We
shall show that $c$ is in fact the coequalizer of $u_{\msign}$ and $u_{\psign}$.

For the moment, let us write $C$ for this coequalizer. We shall describe its
frame $\Omega C$ as a subframe of $\Omega2^{\omega}$ --
it is the equalizer of the frame homomorphisms $u^\ast_\pm$.
From the Stone space
structure of $2^{\omega}$ we see that $\Omega2^{\omega}$ can be described as
the frame of subsets $U$ of $2^{\ast}$, up-closed under the prefix order, and
such that if $s\psign,s\msign\in U$ then $s\in U$. If $t\in2^{\ast}$ then $\up t$
is the principal upset of $t$, so for $s$ in $2^{\omega}$ we have
$s\vDash\,\up t$ iff $t\sqsubseteq s$.

\begin{proposition}\label{prop:omegaC}
  $\Omega C$ is the frame of those subsets $U\in2^{\omega}$
  satisfying the condition that for all finite sign sequences $s$,%
  \[
  (\exists m)s\psign\msign^{m}\in U\longleftrightarrow(\exists n)s\msign\psign^{n}\in U\text{.}%
  \]
\end{proposition}
\begin{proof}
  We have%
  \[
    u_{\msign}^{\ast}(U)=\{s\mid s\msign\psign^{\omega}\vDash U\}
      =\{s\mid(\exists t\in U)t\sqsubseteq s\msign\psign^{\omega}\}
      =\{s\mid(\exists m)s\msign\psign^{m}\in U\}
  \]
  and similarly for $u_{\psign}^{\ast}(U)$.
  The result is now immediate from the fact that $U\in\Omega2^{\omega}$ is in $\Omega C$
  iff $u_{\msign}^{\ast}(U)=u_{\psign}^{\ast}(U)$.
\end{proof}

Having identified $\Omega C$ concretely, our task is now to show that it is
isomorphic to $\Omega\interval$.
The next definition defines two decidable relations on $2^\ast$
that capture (see Proposition~\ref{prop:lmidMidl}) properties of $c'$ and $c$.
For example, $s\lmid t$ holds if, for any stream extending $t$,
we have $c'(s) < c(t)$.

\begin{definition}\label{def:ltBar}
  If $s,t\in2^{\ast}$ then we write $s\lmid t$ if either $s<t$,
  or there is some $k$ with $s\psign\msign^{k}\psign \sqsubseteq t$.

  We write $t\midl s$ if either $t<s$,
  or there is some $k$ with $s\msign\psign^{k}\msign \sqsubseteq t$.
\end{definition}

In other words, for $s\lmid t$ either at the first difference $s$ has $\msign$ and
$t$ has $\psign$, or $s\sqsubseteq t$ and $t$ has $\psign$ immediately after $s$, and at
least one more $\psign$ somewhere further along.

\begin{proposition}\label{prop:lmidMidl}
Let $s,t\in2^{\ast}$.
Then $\up t\leq c^{\ast}((c'(s),1])$ iff $s\lmid t$,
and $\up t\leq c^{\ast}([-1,c'(s)))$ iff $t\midl s$.
\end{proposition}

\begin{proof}
We prove only the first part, since the second follows by interchanging $\psign$
and $\msign$. Using Proposition~\ref{prop:cStar} and the compactness of $\up
t$, we see that $\up t\leq c^{\ast}((c'(s),1])$ iff $\up
t\leq \rup(s\psign\msign^{k}\psign)$ for some $k$, and this clearly holds iff
$s\lmid t$.
\end{proof}

\begin{proposition}
$\Omega C$ is the image of $c^{\ast}$.
\end{proposition}

\begin{proof}
Since $c$ composes equally with $u_{\psign}$ and $u_{\msign}$, we know that it factors
via $C$ and so $\Omega C$ contains the image of $c^{\ast}$.

We show that if $U\subseteq2^{\ast}$ satisfies the condition of
Proposition~\ref{prop:omegaC}, then it is a join of images under $c^{\ast}$ of
dyadic open intervals in $\interval$.

Let $u\in U$. If $u=\varepsilon$ is empty then by up-closure $U=2^{\ast
}=c^{\ast}(\interval)$.

Next, suppose $u=\psign^{n}$ for some $n\geq1$. By the condition on $U$, we find
$s=\psign^{n-1}\msign\psign^{m}\in U$ for some $m$. Then $s\lmid u$; we show that
$\{t\in2^{\ast}\mid s\lmid t\}\subseteq U$. Suppose $s\lmid t$. If $s$ and $t$
disagree, it must be at the $\msign$ in $s$, so $u\sqsubseteq t$ and $t\in U$. On
the other hand, if $s\sqsubseteq t$ then again $t\in U$. The case where
$u=\msign^{n}$ is similar.

Now suppose $u$ contains both $\psign$ and $\msign$. By symmetry it suffices to consider
the case where $U$ ends in $\msign$: so we can write $u=u'\psign\msign^{n}$ with
$n\geq1$. By the condition on $U$ we can find $s_{0}=u'\msign\psign^{m}\in U$
and also $s_{1}=u'\psign\msign^{n-1}\psign\msign^{k}\in U$.
We have $s_{0}\lmid u\midl s_{1}$.
Suppose $s_{0}\lmid t\midl s_{1}$.
If $s_{0}\sqsubseteq t$ or $s_{1}\sqsubseteq t$ then $t\in U$.
Thus we assume $s_{0}<t<s_{1}$.
It cannot disagree with $u'$,
since in its disagreement it would have to have both $\psign$ and $\msign$.
Hence $u'\sqsubseteq t$.
The disagreement with $s_{0}$ must therefore be at the $\msign$ immediately after $u'$.
It follows that $t$ agrees with $s_{1}$ at the first $\psign$ after $u'$,
so the disagreement must be at the second.
Hence $u=u'\psign\msign^{n}\sqsubseteq t$ and $t\in U$.
\end{proof}

After Theorem~\ref{thm:PropSurjn} we can now conclude --

\begin{theorem}
$c\colon 2^{\omega}\rightarrow\interval$ is the coequalizer of
$u_{\pm}\colon 2^{\ast}\rightrightarrows2^{\omega}$.
\end{theorem}

\section{$\interval$ is an interval object in $\Loc$}\label{sec:Initiality}

Let $A$ be an iterative midpoint algebra equipped with points $a_{\pm}$. We
shall also write%
\begin{align*}
a_{0}  &  =m(a_{\msign},a_{\psign})\\
a_{\pm/2}  &  =m(a_{0},a_{\pm})\text{.}%
\end{align*}

If $N\colon \interval\rightarrow A$ as in Definition~\ref{def:intervalObject}, then
$Nc\colon 2^{\omega}\rightarrow A$ is the map $M=M_{a_{\msign}a_{\psign}}$, for%
\[
  Nc(\pm s)=Nm(\pm1,c(s))=m(N(\pm1),Nc(s))=m(a_{\pm},Nc(s))\text{.}%
\]
We can define $M$ regardless of $N$, so it therefore remains to prove (i) that
$M$ factors via $\interval$, as $M=Nc$ for some $N\colon \interval\rightarrow A$,
and (ii) that $N$ is then a midpoint algebra homomorphism.

\begin{lemma}\label{lem:Mpmomega}
  $M(\pm^{\omega})=a_{\pm}$.
\end{lemma}
\begin{proof}
  By the defining property of $M$,
  $M(\pm^{\omega})$ is a point $x_{\pm}$ such that $m(a_{\pm},x_{\pm})=x_{\pm}$.
  But by considering the maps $a_{\pm}\colon 1\rightarrow A$ and $!\colon 1\rightarrow1$ as $h$ and $t$ in
  Definition~\ref{def:convexBody},
  we see that there is a unique map $x_{\pm}\colon 1\rightarrow A$ such that $m(a_{\pm},x_{\pm})=x_{\pm}$.
  Since $a_{\pm}$ satisfies this condition, we deduce $x_{\pm}=a_{\pm}$.
\end{proof}

\begin{proposition}
  $M$ composes equally with $u_{\pm}\colon 2^{\ast}\rightarrow2^{\omega}$.
\end{proposition}
\begin{proof}
  From Lemma~\ref{lem:Mpmomega} we have
  $M(\psign\msign^{\omega})=m(a_{\psign},a_{\msign})=m(a_{\msign},a_{\psign})=M(\msign\psign^{\omega})$, i.e. $M(u_{\psign}(\varepsilon))=M(u_{\msign}(\varepsilon))$.
  It now follows by induction on the length of $s$ that $M(u_{\psign}(s))=M(u_{\msign}(s))$ for
  all $s\in2^{\ast}$.
\end{proof}

It follows that $M$ factors via $\interval$, as $Nc$ for some unique
$N\colon \interval\rightarrow A$.

It remains to be shown that $N$ preserves midpoints, i.e. that $m(N\times
N)=Nm$. Since $c$ is a proper surjection, so too is $c\times c$ and so it
suffices to show that $m(Nc\times Nc)=m(M\times M)=Nm(c\times c)\colon 2^{\omega
}\times2^{\omega}\rightarrow A$.

\begin{definition}
  $\half\colon 2^{\omega}\rightarrow2^{\omega}$ is defined by%
  \[
    \half(\pm s)=\pm\mp s\text{.}%
  \]
\end{definition}

\begin{lemma}\label{lem:half}%
\[
  M_{a_{\msign}a_{\psign}}\half s=m(a_{0},M_{a_{\msign}a_{\psign}}s)\text{.}%
\]
\end{lemma}
\begin{proof}%
  \begin{align*}
    M_{a_{\msign}a_{\psign}}\half(\pm s)  &  =M_{a_{\msign}a_{\psign}}(\pm\mp s)\\
    &  =m(a_{\pm},m(a_{\mp},M_{a_{\msign}a_{\psign}}s))\\
    &  =m(m(a_{\pm},a_{\mp}),m(a_{\pm},M_{a_{\msign}a_{\psign}}s))\\
    &  =m(a_{0},M_{a_{\msign}a_{\psign}}(\pm s))\text{.}%
  \end{align*}
\end{proof}

\begin{lemma}\label{lem:NPreservesmX}
  As maps from $\interval$ to $A$, we have
  \begin{enumerate}
  \item
    $Nm\langle\pm1,\interval\rangle=m\langle a_{\pm},A\rangle N$,
  \item
    $Nm\langle0,\interval\rangle=m\langle a_{0},A\rangle N$.
  \end{enumerate}
\end{lemma}

\begin{proof}
Since $c$ is a surjection, it suffices to show equality when these are
composed with $c$.

(1)%
\begin{align*}
Nm\langle\pm1,\interval\rangle c(s)  &  =Nm(\pm1,c(s))=Nc(\pm s)=M(\pm s)\\
&  =m(a_{\pm},M(s))=m(a_{\pm},Nc(s))=m\langle a_{\pm},A\rangle Nc(s)\text{.}%
\end{align*}

(2)%
\begin{align*}
Nm\langle0,\interval\rangle c(s)  &  =Nm(0,c(s))\\
&  =Nc\half(s)\text{ (by Lemma~\ref{lem:half}, using }c=M_{-1,+1}%
\text{)}\\
&  =M\half(s)\\
&  =m(a_{0},M(s))\text{ (by Lemma~\ref{lem:half} again, using }%
M=M_{a_{\msign}a_{\psign}}\text{)}\\
&  =m\langle a_{0},A\rangle Nc(s)\text{.}%
\end{align*}
\end{proof}

To analyse preservation of midpoints we shall need to define a version of the
midpoint function that works entirely on sign sequences. However, it will
convenient to use sequences that may include $0$: so we shall use $3^{\omega}$
where we take $3=\{\psign,\msign,0\}$.
There is an obvious inclusion $i\colon 2^{\omega}\rightarrow3^{\omega}$.

We define $M_{0}\colon 3^{\omega}\rightarrow A$, similar to $M$,
but with the additional condition that $M_{0}(0s)=m(a_{0},M(s))$.
In other words, in Definition~\ref{def:itMidptAlg} the head map
$h \colon 3^\omega \to \interval$ takes $0s$ to $a_0$.
Then clearly $M=M_{0}i$.

We can do the same with $c$ instead of $M$, obtaining a unique map
$c_{0}\colon 3^{\omega}\rightarrow\interval$ such that $c_{0}(\pm s)=m(\pm
1,c_{0}(s)),c_{0}(0s)=m(0,c_{0}(s))$. Then $c=c_{0}i$.

\begin{lemma}\label{lem:M0}
  $M_{0}=Nc_{0}$.
\end{lemma}
\begin{proof}%
  \[
    Nc_{0}(\pm s)=Nm(\pm1,c_{0}s)
      =m(a_{\pm},Nc_{0}s)\text{ (Lemma~\ref{lem:NPreservesmX} (1))}%
  \]%
  \[
    Nc_{0}(0s)=Nm(0,c_{0}s)=m(a_{0},Nc_{0}s)\text{ (Lemma~\ref{lem:NPreservesmX} (2))}%
  \]
  It follows that $Nc_{0}$ has the characterizing property of $M_{0}$.
\end{proof}

\begin{definition}
The \emph{sequence midpoint map} $m_{s}\colon 2^{\omega}\times2^{\omega}%
\rightarrow3^{\omega}$ is defined by%
\begin{align*}
m_{s}(\pm s_{1},\pm s_{2})  &  =\pm m_{s}(s_{1},s_{2})\\
m_{s}(\pm s_{1},\mp s_{2})  &  =0m_{s}(s_{1},s_{2})\text{.}%
\end{align*}
\end{definition}

\begin{lemma}\label{lem:ms}
  $m(M\times M)=M_{0}m_{s}$.
\end{lemma}
\begin{proof}
  They are both the unique map $f\colon 2^{\omega}\times2^{\omega}\rightarrow A$
  such that $f(\pm s_{1},\pm s_{2})=m(a_{\pm},f(s_{1},s_{2}))$
  and $f(\pm s_{1},\mp s_{2})=m(a_{0},f(s_{1},s_{2}))$.
  For $m(M\times M)$,%
  \begin{align*}
    m(M\times M)(\pm s_{1},\pm s_{2})
      &  =m(m(a_{\pm},M(s_{1})),m(a_{\pm},M(s_{2})))\\
      &  =m(a_{\pm},m(M\times M)(s_{1},s_{2}))\text{,}\\
    m(M\times M)(\pm s_{1},\mp s_{2})
      &  =m(m(a_{\pm},M(s_{1})),m(a_{\mp},M(s_{2})))\\
      &  =m(m(a_{\pm},a_{\mp}),m(M(s_{1}),M(s_{2})))\\
      &  =m(a_{0},m(M\times M)(s_{1},s_{2}))\text{.}%
  \end{align*}
  For $M_{0}m_{s}$,%
  \begin{align*}
    M_{0}m_{s}(\pm s_{1},\pm s_{2})  &  =M_{0}(\pm m_{s}(s_{1},s_{2}))\\
    &  =m(a_{\pm},M_{0}m_{s}(s_{1},s_{2}))\text{,}\\
    M_{0}m_{s}(\pm s_{1},\mp s_{2})  &  =M_{0}(0m_{s}(s_{1},s_{2}))\\
    &  =m(a_{0},M_{0}m_{s}(s_{1},s_{2}))\text{.}%
  \end{align*}
\end{proof}

\begin{corollary}\label{cor:ms}
  $m(c\times c)=c_{0}m_{s}$.
\end{corollary}
\begin{proof}
  Replace $A$ by $\interval$.
\end{proof}

\begin{proposition}
  $N\colon \interval\rightarrow A$ preserves midpoints.
\end{proposition}
\begin{proof}%
  \begin{align*}
    m(N\times N)(c\times c)
      &  =m(M\times M)=M_{0}m_{s}\text{ (Lemma~\ref{lem:ms})}\\
      &  =Nc_{0}m_{s}\text{ (Lemma~\ref{lem:M0})}\\
      &  =Nm(c\times c)\text{ (Corollary~\ref{cor:ms}).}%
  \end{align*}
  We now use the fact that $c\times c$ is a surjection, following from the fact
  that $c$ is a proper surjection.
\end{proof}

Putting together all the results of this section, we obtain --

\begin{theorem}
  $\interval=[-1,1]$ is an interval object in the category $\Loc$ of locales.
\end{theorem}

\section{Conclusions}\label{sec:Conc}

The main result was about $\interval$ as interval object,
but along the way we also showed that the map $c\colon 2^\omega \to \interval$,
evaluating infinite binary expansions,
is a proper localic surjection that is easily expressed as a coequalizer.
This result has some interest in itself.
In classical topology, $c$ is a surjection because for every Dedekind section
there is an infinite expansion; however, this uses choice.
Essentially, the surjectivity of $c$, in other words the monicity of $c^\ast$,
is a conservativity result, and this is known as a constructive substitute
for using choice to find the existence of points.
See, for example, the constructive Hahn-Banach Theorem in \cite{MulvPellet}.
However, our result is unusual in using a \emph{proper} surjection rather than an \emph{open} one.

The proof of proper surjectivity used the preframe coverage theorem in a standard way.
However, it was more intricate than I expected.
I had a hope to use the metric space theory again for $2^\omega$,
but was put off by the fact that to get $2^{\omega}$ as a completion of $2^{\ast}$
requires each finite sequence $s$ to be identified with an infinite sequence,
either $s\msign^{\omega}$ or $s\psign^{\omega}$: this breaks symmetry.
I conjecture there's a way forward using partial metrics,
so that $2^{\ast}$ is metrized with $d(s,s)=2^{1-|s|}$.
However, we do not at present have a theory of localic completion of partial metrics.
It would be easier with $c_{0}\colon 3^{\omega}\rightarrow\interval$,
but then that would presumably make Section~\ref{sec:Coequ} harder.
In any case, the result with $2^\omega$ is stronger.

The main result, on $\interval$ as an interval object, free on two points,
suggests generalization to simplices, free on their vertices.
I conjecture that similar techniques to prove this, using infinite sequences,
could be developed using barycentric subdivision.

\bibliographystyle{amsplain}
\bibliography{MyBiblio}
\end{document}